\documentclass[leqno,12pt]{article}
\usepackage{latexsym}
\usepackage{amsmath}
\usepackage{amssymb}
\usepackage{bbm}
\usepackage{amsthm}
\usepackage{holtpolt}
\usepackage[utf8]{inputenc}
\usepackage[authoryear]{natbib}
\usepackage{graphicx}
\usepackage{epsfig}

 0 \DeclareMathAlphabet{\mathbfit}{OT1}{cmr}{bx}{it} \DeclareRobustCommand{\vec}[1]{\mathbfit{#1}}

\newcommand{\bea}{\begin{eqnarray*}}
\newcommand{\eea}{\end{eqnarray*}}
\newcommand{\be}{\begin{eqnarray}}
\newcommand{\ee}{\end{eqnarray}}
\newcommand{\beq}{\begin{equation}}
\newcommand{\eeq}{\end{equation}}

\def\inte{\operatorname{Int}}

\newtheorem{thm}{Theorem}[section]
\newtheorem{cor}[thm]{Corollary}
\newtheorem{lem}[thm]{Lemma}
\newtheorem{prop}[thm]{Proposition}

\newtheorem{rem}[thm]{Remark}
\renewenvironment{proof}{{\itshape Proof:}}{\\ \qed \\}

 \def\er{\mathbb{R}}
 \def\en{\mathbb{N}}

\usepackage{hyperref}
\hypersetup{
    final,
    colorlinks=true,
    raiselinks=true, 
    linkcolor=[rgb]{.0, .0, .8}, 
    citecolor=[rgb]{.0, .0, .8}, 
    filecolor=black, 
    urlcolor=black 
}

\begin{document}
\nocite{*}

\title{ Hankel determinants of random moment sequences }

\author{
{\small Holger Dette} \\
{\small Ruhr-Universit\"at Bochum} \\
{\small Fakult\"at f\"ur Mathematik} \\
{\small 44780 Bochum, Germany} \\
{\small e-mail: holger.dette@rub.de}\\
\and
{\small Dominik Tomecki}\\
{\small Ruhr-Universit\"at Bochum }\\
{\small Fakult\"at f\"ur Mathematik}\\
{\small 44780 Bochum, Germany}\\
{\small e-mail: dominik.tomecki@rub.de}\\
}

\maketitle

\begin{abstract}
For $ t \in [0,1]$ let  $\underline{H}_{2\lfloor  nt \rfloor} = ( m_{i+j})_{i,j=0}^{\lfloor nt \rfloor} $  denote the  Hankel matrix
of order $2\lfloor  nt \rfloor$ of a random vector
$(m_1,\ldots ,m_{2n})$  on the moment space $\mathcal{M}_{2n}(I)$
of all moments  (up to the order $2n$) of probability measures  on the interval $I \subset \er $. In this paper we study
the asymptotic properties  of the stochastic process $\{  \log \det \underline{H}_{2\lfloor  nt \rfloor} \}_{t\in [0,1]}$
as $n \to \infty$. In particular weak convergence and corresponding large deviation principles are
derived after appropriate standardization.

\end{abstract}

\medskip



\section{Introduction} \label{sec1}
 \def\theequation{1.\arabic{equation}}
 \setcounter{equation}{0}

Hankel matrices  are well studied objects in mathematics with applications in various fields
such as orthogonal polynomials, random matrices or operator theory. Asymptotic properties
of functions  of non-random Hankel matrices
such as  the determinant,  condition number or smallest eigenvalue have been
studied by \cite{hirschman1966}, \cite{zamar2001}, \cite{basor2001} or \cite{bergszwarc2011} among others. Recently, random
Hankel matrices have also been considered in the literature with the main focus   on matrices with independent entries.
For example,
 \cite{bryc2006} studied the limiting spectral measure of large Hankel (and Toeplitz) matrices, while
some results regarding the operator norm can be found in \cite{bose2007}. \\
The present paper takes a different look at random Hankel matrices (more precisely, at their log-determinants) with 
not necessarily independent entries. Our investigations are motivated by the fact that Hankel matrices are usually used to characterize the solution of classical moment problems.
To be precise,
let   $I \subset \mathbb{R}$ denote an interval and define  $ \mathcal{P} (I) $  as  the set of all probability measures on the Borel field of
$I$ with existing moments. For a measure $\mu \in \mathcal{P} (I)$ we denote by
$$ m_k=  m_k(\mu) = \int_I x^k \mu (dx) \ ; \qquad k=0,1,2,\dots  $$
the $k$-th moment   and define
  \be
   \label{mominfty}
\mathcal{M} (I) = \left \{ \vec{m}(\mu) = (m_1(\mu),m_2(\mu),\dots )^t | \ \mu \in \mathcal{P} (I) \right \} \subset \mathbb{R}^{\mathbb{N}} \ee
as the set of all moment sequences. We denote by $\Pi _n$ ($n \in \mathbb{N}$) the
canonical projection onto the first $n$ coordinates and call
$\mathcal{M}_n (I) = \Pi_n \left( \mathcal{M} (I) \right) \subset \mathbb{R}^n$
the $n$-th moment space. The Hamburger moment problem is to decide
if  a given sequence $(m_n)_{n \in \mathbb{N}}$ is an element of $\mathcal{M} (\mathbb{R})$ and it is well known that 
this is the case if and only if the  Hankel matrices  $\underline{H}_{2k}=   (m_{i+j})^k_{i,j=0}$  are nonnegative definite for all
$k\in \en$ [see \cite{shohat1943}]. Moreover the vector 
 ${\vec m}_{2n}= (m_1,\ldots , m_{2n})$ is an element of the moment space  $  \mathcal{M}_{2n} (\mathbb{R})$ if and only if  the Hankel 
 matrix   $\underline{H}_{2n}=   (m_{i+j})^n_{i,j=0}$
 is nonnegative definite. Similar characterization can be obtained for the Stieltjes and
     Hausdorff  moment problem corresponding to measures on the half line $\mathbb{R}^+_0=[0,\infty)$ and the interval $[0,1]$, respectively.
\\
 \cite{chakemstu1993} considered the ``classical'' moment space
corresponding to measures on the interval $[0,1]$ [see \cite{karsha1953}, \cite{krenud1977}, for  some early references]
and  equipped $\mathcal{M}_n ([0,1]) $   with a uniform distribution. They   proved asymptotic normality of an appropriately
standardized version of a projection $\Pi_k ({\vec m_n})$ of a  uniformly distributed vector ${\vec m_n}$ on $\mathcal{M}_n([0,1])$
as $n\to \infty $.
\cite{gamloz2004}  investigated corresponding large deviation principles, while  \cite{lozada2005}
 studied similar problems for moment spaces corresponding to more general functions defined on a bounded set. More
 recently,  some of these  results have been generalized by \cite{detnag2012}  to the moment spaces $\mathcal{M}_n ([0,\infty))$ and $\mathcal{M}_n(\mathbb{R})$ corresponding to unbounded intervals.  \\
The present paper is devoted to the asymptotic analysis of Hankel determinants of random   moment vectors on $\mathcal{M}_{2n}(I)$. For example, if $
 \vec{m}_{2n}=
 (m_1,\dots,m_{2n})$ denotes a random vector uniformly distributed on the $2n$th moment space $\mathcal{M}_{2n}([0,1])$, then it is shown in this paper that an appropriately transformed and standardized version of the determinant of the random Hankel matrix  $\underline{H}_{2n}= (m_{i+j})^n_{i,j=0}$ converges weakly, that is
\begin{equation} \label{1.1}
\frac {2} {\sqrt{n}} \big \{ \log \det \underline{H}_{2n} - \log \det \underline{H}^0_{2n} + \frac {n}{2} \big\}  \stackrel{\mathcal{D}}{\longrightarrow} \mathcal{N} (0,1),
\end{equation}
where $\underline{H}^0_{2n}= (m^0_{i+j})^n_{i,j=0}$ denotes the Hankel determinant of the moments of the arcsine distribution on the interval $[0,1]$,
that is $m_\ell = (\begin{smallmatrix}
                                  2 \ell \\
                                  \ell
 \end{smallmatrix})2^{-2 \ell}$. Moreover, the sequence
 \begin{equation} \label{1.1a}
- \frac {1} {n} \big \{ \log \det \underline{H}_{2n} - \log \det \underline{H}^0_{2n}  \big\}
\end{equation}
satisfies a large deviation principle with a good rate function. It will be demonstrated in Section \ref{sec2} that the moments of the arcsine distribution used for the centering
in \eqref{1.1} and \eqref{1.1a} correspond to the center of the moment space  $\mathcal{M}([0,1])$. \\
Similar results are available for the moment spaces $\mathcal{M}_{2n}([0, \infty))$ and $\mathcal{M}_{2n}(\mathbb{R})$, where the centering has to be performed by the logarithms of the determinants of the Hankel matrices corresponding to the Marcenko-Pastur law and Wigner's semi-circle law, respectively
(in these cases the corresponding Hankel determinants $\underline{H}_{2n} $ have  value $1$).
These measures play a very important role in
the theory of random matrices, free probability and quantum probability, see the books of \cite{hiaipetz2000} and \cite{horaobata2007} among others. \\
The remaining part of this paper is organized as follows. In Section \ref{sec2} we present some  facts on moment theory and introduce random moment sequences on the spaces $\mathcal{M}_n ([0,1]),  \mathcal{M}_n ([0,\infty))$ and $\mathcal{M}_{2n - 1} (\mathbb{R})$.
 We also state some basic properties of these random variables which will be useful in the following discussion. In Section \ref{sec3} it is shown that
for the canonical distributions on the moment space $\mathcal{M}_{2n} (I)$ an  appropriately standardized version of the stochastic process
\be
\{  D_{2\lfloor  nt \rfloor} \}_{t\in [0,1]} =\{   \log \det \underline{H}_{2\lfloor  nt \rfloor} \}_{t\in [0,1]}
\ee
converges weakly to a Gaussian process. The centering and  scaling is different for the three moment spaces under consideration.  
We also study the asymptotic properties of the vector $(D_{n, 2}, \ldots, D_{n, 2k})^t$ for any fixed $k$. 
Large deviation principles are investigated in Section \ref{sec4}, while some technical results which are required for the proofs are provided in the Appendix.

\section{Some basic facts about  moment theory} \label{sec2}
 \def\theequation{2.\arabic{equation}}
 \setcounter{equation}{0}

 Similar to cumulants, canonical moments provide a one-to-one transformation of the ordinary moments. They appear naturally in the continued fraction expansion of the Stieltjes transform of a probability measure but are less known than cumulants. Therefore, we state some basic facts in the following two paragraphs, where we distinguish between bounded and unbounded intervals.

 \subsection{Canonical moments} \label{sec21}

Canonical moments  have been investigated in a series of papers by  \cite{skibinsky1967,skibinsky1968,skibinsky1969} and roughly speaking define a
one-to-one mapping from the set of moments $ \mathcal{M} ([0,1])  $  (or more generally from $\mathcal{M} ([a,b])$  for any finite
 interval $[a,b] \subset \er $) onto the set $[0,1]^\en$.  They have implicitly been discussed before in the work
 of \cite{verblunsky1935,verblunsky1936}, who mainly considered measures on the unit circle. In this section we  briefly present some  basic facts   for the sake of a self contained presentation and discuss corresponding results for the set $\mathcal{M} ([0,\infty))$ and $\mathcal{M}(\mathbb{R})$. For details we refer to
 the  monographs of  \cite{dettstud1997} and \cite{wall1948}.
For a given vector  ${\vec m_{k-1}=}(m_1,\dots ,m_{k-1})^T \in \mathcal{M}_{k-1}([0,1])$  of moments of a probability measure on the interval $[0,1]$ define
\be
\nonumber
m_{k}^- = \min \left \{ m_{k}(\mu) \left| \mu \in \mathcal{P}([0,1]) \mbox{ with } \int_0^1 t^i d\mu(t) = m_i \mbox{ for } i = 1,\dots ,k-1 \right. \right \}, \\
\nonumber m_{k}^+ = \max \left \{ m_{k}(\mu) \left| \mu \in \mathcal{P}([0,1]) \mbox{ with } \int_0^1 t^i d\mu(t) = m_i \mbox{ for } i
= 1,\dots ,k-1 \right. \right \}.
 \ee
 Throughout this paper let $\inte C$ denote the interior of a set $C$.
 It is shown in \cite{dettstud1997}
  that $ \vec{m}_k = (m_1,\dots ,m_{k})^T \in \inte \mathcal{M}_{k}([0,1]) $ if and only if
$ m_{k}^- <m_k < m_{k}^+$. In this case    the  canonical moments of order $l=1,\dots,k$ are defined  as
\be \label{canmom} p_l =p_l(\vec{m}_k) = \frac{m_{l}-m_{l}^-}{m_l^+-m_l^-} \ ; \qquad l=1,\dots,k. \ee
 Note that for ${\vec m_k}\in \inte \mathcal{M}_k ([0,1])$ we have $p_l\in(0,1); \  l=1,\dots,k$; and that $p_k$ describes
  the relative position
 of the moment $m_k$ in the set of all possible $k$-th moments with fixed moments $m_1, \dots,m_{k-1}$. It can also be shown
that  the definition \eqref{canmom} defines a one-to one mapping from $\inte \mathcal{M}_n ([0,1])$
 onto the open cube $(0,1)^n$.  As an example consider the arcsine distribution  $\mu^0$  on the interval $[0,1]$ with density $1/(\pi \sqrt{x(1-x)})$, then
the corresponding canonical moments are given by $p_\ell=1/2$ for all $\ell \in \en$ [see \cite{dettstud1997}]. Consequently, the sequence of
 moments of the arcsine distribution defines the center of the moment  space $ \mathcal{M}([0,1]) $. Note however, that it is not the barycenter of the moment space.\\
The determinant of the Hankel matrix $\underline{H}_{2k} =(m_{i+j})_{i,j=0}^k$
 of the moment vector $(m_1,\ldots , m_{2k})$ can easily   be
 expressed in terms of the corresponding canonical moments, that is
 \begin{eqnarray}\label{detbound}
 \det \underline{H}_{2k}  =  \det (m_{i+j})_{i,j=0}^k =  \big( p_1q_1p_2 \big)^k \prod_{j=2}^k  \big( q_{2j-2}  p_{2j-1}q_{2j-1}p_{2j} \big)^{k-j+1}~,
\end{eqnarray}
where  $q_j=1-p_j$ [see  \cite{dettstud1997}, Theorem 1.4.10]. \\
In the case $I=[0,\infty)$
the upper bound $m_k^+$  is in general not finite, but we can still define for a point ${\vec m_{k-1}} \in \inte \mathcal{M}_{k-1} ([0,\infty))$ the lower bound
\bea
 \label{defm-2}
m_{k}^- = \min \left \{ m_{k}(\mu) \left| \mu \in \mathcal{P}([0,\infty)) \mbox{ with } \int_0^{\infty} t^i d\mu(t) = m_i \mbox{ for } i = 1,\dots
,k-1 \right. \right \}, \eea where ${\vec m_k}=(m_1,\dots ,m_{k})^T \in \inte \mathcal{M}_{k}([0,\infty)) $ if and only if $  m_{k} > m_{k}^-$ . In
this case,  the analogues of the canonical moments are defined by the quantities \be\label{defz} z_l &= &
\frac{m_{l}-m_{l}^-}{m_{l-1}-m_{l-1}^-} \qquad l=1,\dots,k \ee (with $m_0^-=0$).
As in the case of a bounded interval  the definition  \eqref{defz} provides a one to one mapping
from $  \inte \mathcal{M}_n([0,\infty)) $  onto $(\er^+)^n$, and it can be shown using similar arguments as in \cite{dettstud1997} that 
$$
\det \underline{H}_{k} = (m_{k} - m_{k}^-)   \det \underline{H}_{k - 2}~,~k \ge 2.
$$
Consequently, the determinant of the Hankel matrix is given by
\begin{eqnarray}\label{detunbound}
 \det \underline{H}_{2k}  =  \det (m_{i+j})_{i,j=0}^k =  \prod_{j=1}^k  \big( z_{2k-1}z_{2k} \big)^{k-j+1}~,
\end{eqnarray}
Finally, in the case $I = \mathbb{R}$ neither $m_{k}^- $
nor $m_{k}^+$ are in general finite. Nevertheless, there exists an analogue of the quantities $p_i$ and $z_i$ defined in \eqref{detbound} and \eqref{detunbound}.
To be precise, we define for a vector $ {\vec m_{2n-1}}  = (m_1,\dots,m_{2n-1})   \in  \mathcal{M}_{2n-1}(\er)$ 
with $\underline{H}_{2n-2} >0$  the polynomial
\be \label{monpol} P_{k} (x) ~=~ \left|\begin{array}{cccc}
  m_0& \cdots& m_{k-1}& 1\\
  \vdots& \ddots & \vdots& \vdots\\
  m_k& \cdots& m_{2k-1}& x^k
  \end{array}\right| \Big/
\left|\begin{array}{ccc}
  m_0& \cdots& m_{k-1}\\
  \vdots&  \ddots& \vdots\\
  m_ {k-1}& \cdots& m_{2k-2}
  \end{array}\right| \ ; \quad k = 1,  \dots, n
\ee
 [see  \cite{chihara1978}]. We consider  a one to one mapping
 \be \label{defxi} \xi_n : \left\{ \begin{array}{l}
                        \inte \mathcal{M}_{2n-1}(\er) \longrightarrow \left( \er \times \er^+ \right)^{n-1} \times \er  \\
                         {\vec m_{2n-1}} \mapsto (b_1,a_1,\dots ,a_{n-1},b_n)^T  \end{array} \right.~
\ee
 defined by
 \be
\int_\mathbb{R} x^k P_k(x) d\mu(x) &=&  a_1\dots a_k;  \qquad k=1,\dots,n-1  \label{wall1},  \\
\int_\mathbb{R} x^{k+1} P_k(x) d\mu(x) &=&  a_1\dots a_k (b_1 + \dots +b_{k+1}); \qquad k=0, \dots,n-1 \label{wall2} ,
 \ee
 where $\mu $ is any measure with first $2n-1$ moments given by  $ (m_1,\dots,m_{2n-1}) $ [see for example  \cite{wall1948}]. Note that $P_1(x),\dots, P_n(x)$ are orthogonal polynomials with leading coefficient $1$ with respect to the measure $\mu$.
 It is now easy to see that the determinant of the Hankel matrix can be represented as
 \begin{eqnarray}\label{detunbound2}
 \det \underline{H}_{2k}  =  \det (m_{i+j})_{i,j=0}^k =  \prod_{j=1}^k    a_j^{k-j+1}~.
\end{eqnarray}
In the following section we will equip these moment spaces with distributions. We begin with the moment space
corresponding to measures on bounded intervals.

\subsection{Distributions on moment spaces} \label{sec22}

\cite{chakemstu1993} considered
 a uniformly distributed vector on the set $\mathcal{M}_n ([0,1])$ and  showed that  an appropriately
standardized version of a projection $\Pi_k ({\vec m_n})$
onto  its first $k$ components is asymptotically normal distributed, where the centering has to be performed with the moments of the arcsine distribution. A key ingredient in their proof is the following lemma, which shows
  that the canonical moments of  a  uniformly distributed vector ${\vec m_n}$ on $\mathcal{M}_n([0,1])$ are independent
[for a proof  see \cite{dettstud1997}]. For this 
and the following statements we will make the dependence of the canonical moments
on the dimension of the moment space $\mathcal{M}_n ([0,1])$ more explicit. More precisely, we use the notation $p_{n,\ell} (\vec{m}_n) $ instead of $p_\ell (\vec{m}_n) $,
and  the symbol $\beta (a,b) $ denotes a Beta-distribution
on the interval $[0,1]$ with  density 
$$  I_{[0,1]}(x) x^{a - 1} (1 - x)^{b - 1} /  {B(a, b)} .$$
\begin{lem} \label{betadist}
For a  uniformly distributed random vector ${\vec m_n}$ on the $n$th moment space 
$\mathcal{M}_n ([0,1])$  the canonical moments $p_{n,1}({\vec m_n}), \ldots , p_{n,n}({\vec m_n})$
defined by \eqref{canmom} are independent
and Beta-distributed, that is
$$
p_{n,i} ({\vec m_n} ) \sim \beta (n-i+1,n-i+1)~.
$$
\end{lem}
Note that the mapping between the (regular) moments and the canonical moments has only been defined on the 
the interior   of $\mathcal{M}_n([0, 1])$. However, $\mathcal{M}_n([0, 1])$ is a closed, convex set and therefore its 
boundary has Lebesgue measure $0$. Since we endow this space with the uniform distribution, the random variables $p_{n, i}$ are a.s. well-defined.
We also note that  \cite{detnag2012} defined  more general  distributions on $\mathcal{M}_n([0, 1])$, which contain the uniform distribution as
a special case.
\\
In order to define  an analogue of the uniform distribution on the unbounded moment space   $\mathcal{M}_n ([0,\infty))$ these authors
use the relation \eqref{defz}. To be precise, consider a random vector $\vec{m}_n$ and denote the quantities in \eqref{defz}
by $z_{n,1}(\vec{m}_n), \ldots , z_{n,n}(\vec{m}_n)$. A density on the moment space $\mathcal{M}_n ([0,\infty))$ is then defined by
 \be\label{momdichte0infty}
g_n^{(\gamma,\delta)}(\vec{m}_n)
                                                               &=& c_n^{[0,\infty )} \prod_{k=1}^n z_{n,k}(\vec{m}_n)^{\gamma_{n,k}} \exp (-\delta_{n,k} z_{n,k}(\vec{m}_n)) \mathbbm{1}_{ \{ z_{n,k}(\vec{m}_n) >0 \} },
\ee
where the constants satisfy    $\gamma_{n,k} > -(n-k+1)$, 
$\delta_{n,k} >0 $ for $k=1, \ldots , n$,   and the
normalizing constant is given by
$c_n^{[0,\infty)} = \prod_{k=1}^n ( {\delta_{n,k}^{\gamma_k +n-k+1}}) /{\Gamma(\gamma_{n,k} +n-k+1)} ). $ The analogue of Lemma \ref{betadist}  is now
provided  by the following result, where the symbol $\gamma (a,b) $ denotes a Gamma distribution ($a,b>0$)  with density
$$
{b^a \over \Gamma (a) } x^{a-1} e^{-bx} \mathbbm{1}_{[0,\infty)} (x) ~.
$$
\begin{lem} \label{gammadist}
For a    random vector ${\vec m_n}$  with density \eqref{momdichte0infty} on  $\mathcal{M}_n ([0,\infty))$
 the canonical moments $z_{n,k}(\vec{m}_n)$
defined by  \eqref{defz} are independent and Gamma-distributed, that is
$$
z_{n,k}(\vec{m}_n) \sim \gamma(\gamma_{n,k}+n-k+1,\delta_{n,k}),\qquad k=1,\dots,n.
 $$
\end{lem}

\smallskip \smallskip

A proof of Lemma \ref{gammadist} can be found in \cite{detnag2012} and we
 conclude this section with the corresponding statements for the moment space $\mathcal{M}_{2n - 1}(\mathbb{R})$.
Following \cite{detnag2012} we define 
a distribution on  $\mathcal{M}_{2n-1} (\mathbb{R})$ by  
\begin{align}  \nonumber
h_{2n-1}^{(\gamma  ,\delta )}(\vec{m}_{2n-1}) =& \prod_{k=1}^n \sqrt{\frac{\delta_{n,2k-1}}{\pi}} \exp \left( -\delta_{n,2k-1} b_{n,k}^2 (\vec{m}_{2n-1}) \right) \\
&                                       \times \prod_{k=1}^{n-1} \frac{{\delta_{n,2k}}^{\gamma_{n,k} +2n-2k}}{\Gamma (\gamma_{n,k} +2n-2k)} a_{n,k}^{\gamma_{n,k}}
(\vec{m}_{2n-1}) \exp \left( -\delta_{n,2k} a_{n,k} (\vec{m}_{2n-1}) \right) \mathbbm{1}_{ \{ a_{n,k} (\vec{m}_{2n-1}) >0 \} } ,
  \label{momdichteer} 
  \end{align}
  where the  constants  satisfy  $\gamma_{n,k} > -2(n-k)$  for  $k=1,\ldots ,n-1$ and  $\ \delta_{n,1}, \dots, \delta_{1,2n-1}> 0$.  
The  distribution of the corresponding quantities $a_k$ and $b_k$ defined by \eqref{wall1} and \eqref{wall2} is specified in the following
result.

\begin{lem} \label{gammadist2}
Let $\vec{m}_{2n-1} \in \mathcal{M}_{2n-1}(\er)$ be a random vector with density $h_{2n-1}^{(\gamma,\delta)}$ defined in
\eqref{momdichteer}.
Then the random  coefficients  $(b_{n,1},a_{n,1}\dots ,a_{n,n-1},b_{n,n})^T $ defined by \eqref{wall1} and \eqref{wall2}   are
independent and
\begin{align*}
b_{n,k} &\sim \mathcal{N} (0,\tfrac{1}{2\delta_{n,2k-1}}), ~a_{n,k} \sim \gamma (\gamma_{n,k} +2n-2k,\delta_{n,2k}).
\end{align*}
\end{lem}

\medskip

\begin{rem} \label{random} {\rm There exists an interesting relation to random matrix theory in particular to the $\beta$-ensembles considered by
\cite{dumede2002,edesut2008,ramridvir2011} among others. To be precise, consider  exemplarily the moment space  $ \mathcal{M}_{2n-1}(\er)$.
It can be shown that for a point $\vec{m}_{2n-1} \in  \inte \mathcal{M}_{2n-1} (\er )$  the polynomials defined in \eqref{monpol}  satisfy the  three term
recurrence relation
\be
\label{polrek} x P_k(x) = P_{k+1}(x) + b_{k+1} P_k(x) + a_k P_{k-1}(x); \qquad k=1,\dots,n-1, 
\ee 
($P_0(x) = 1,P_1(x) = x-b_1$), where the coefficients in the recursion are defined by \eqref{wall1} and \eqref{wall2}.  A straightforward
calculation now shows that the polynomial $P_n(x)$ is the characteristic polynomial  $\det (xI_n -A_n)$ of the matrix  
\begin{equation} \label{dum}
A_n ~=~
\begin{pmatrix}
b_{1} &   \sqrt{a_{1}} \\
  \sqrt{a_{1}} &  {b_{2}} &   \sqrt{a_{2}} \\
 & \ddots & \ddots & \ddots \\
& & \sqrt{a_{n-2}} &   {b_{n-1}} &  \sqrt{a_{n-1}} \\
& & &   \sqrt{a_{n-1}} &  {b_{n}} 
\end{pmatrix}~.
\end{equation} 
If $\vec{m}_{2n-1}$ 
is a random vector on  $ \mathcal{M}_{2n-1}(\er)$ with density  $h_{2n-1}^{(\gamma,\delta)}$ defined in \eqref{momdichteer}, 
and $\delta_{n,2k-1} = 1/2$ ($k=1,\ldots , n$),  $\delta_{n,2k} = 1$, $\gamma_{n,k} = ({1\over 2} \beta -2) (n-k) $   ($k=1,\ldots , n-1$) for some $\beta >0$,
then  it follows from Lemma \ref{gammadist2}
that   the coefficients  in this matrix  are independent with distributions $b_i \sim {\cal N} (0, 1)$, $a_i \sim {1\over 2}  \chi^2_{\beta (n-i)}$. This means that 
$P_n(x) $ is the characteristic polynomial of the random the matrix \eqref{dum} corresponding to the $\beta$- Hermite ensemble as introduced by
\cite{dumede2002}. While the common matrix literature investigates spectral properties  of this matrix, the random Hankel determinant
corresponds to a product of $L^2$-norms  of the (random) polynomials $P_1, \ldots , P_n$, that is
$$
\prod_{i=1}^n \int_\er P_i^2(x) \mu (dx) ~= ~\prod_{i=1}^n a_i^{n-i+1}  ,
$$
where  $\mu$ denotes a random measure whose first $2n-1$ moments are defined the random Jacobi matrix \eqref{dum}. \\
We also note that a similar interpretation is available for the random moment sequences on   $ \mathcal{M}_{2n-1}([0,1])$
and $ \mathcal{M}_{2n-1}([0,\infty) )$ observing the results of  \cite{kilnen2004} and \cite{dumede2002}, respectively.
}
\end{rem}

\section{Weak convergence of Hankel determinant processes} \label{sec3}
 \def\theequation{3.\arabic{equation}}
 \setcounter{equation}{0}

 Throughout this section we investigate the asymptotic properties of the stochastic process
 \be \label{hanproc}
 \{  D_{n,2\lfloor  nt \rfloor}   (\vec{m}_{2n}) \}_{t\in [0,1]} =\{   \log \det \underline{H}_{n,2\lfloor  nt \rfloor} (\vec{m}_{2n})  \}_{t\in [0,1]}
 \ee
 where $\underline{H}_{n,2\lfloor  nt \rfloor} (\vec{m}_{2n}) = (m_{i+j} (\vec{m}_{2n}) )_{i,j=0}^{\lfloor nt \rfloor}$
 is the Hankel determinant of a random vector  $ \vec{m}_{2n}$ on the moment space $\mathcal{M}_{2n}(I)$. We
 also investigate the asymptotic  properties of the vector $(D_{n,2}   (\vec{m}_{2n}) , \ldots, D_{n,2k}   (\vec{m}_{2n}))$ for some fixed $k \in \en $.
 In the following discussion we  
 treat the cases of a bounded and unbounded moment space separately. \\
In the following discussion the symbol $Y_n  \xrightarrow{d}  Y$ denotes weak convergence of a vector valued sequence 
 of random variables 
 $(Y_n)_{n\in \en} $. Moreover,  let
    \begin{align*}
        l^\infty([0, 1]) = \{f:[0, 1] \to \mathbb{R} \mid \|f\|_\infty < \infty\}
    \end{align*}
  denote the space of bounded real-valued functions on the interval  $[0, 1]$, with the topology induced by the uniform norm $\| \cdot \|_\infty$. We
  denote by $  X_n \Longrightarrow X$ the weak convergence  of a sequence $(X_n)_{n\in \en} $ of  random variables
   in  $l^\infty([0, 1])$  to an $l^\infty([0, 1])$-valued random variable $X$. This is the mode of convergence described in Chapter 1.5 of \cite{vaartwell1995}. We also use the convention $0 \log (0) := 0 $ and denote by $ s\wedge t$ resp. $ s \vee t$ the minimum resp. maximum of $s,t \in \mathbb{R}$.

 \subsection{Hankel determinants from $\mathcal{M}_{2n} ([0,1])$} \label{sec31}
 Consider a uniformly distributed  random vector $\vec{m}_{2n}$ on  $\mathcal{M}_{2n} ([0,1])$, that is 
 $\vec{m}_{2n} \sim \mathcal{U}(\mathcal{M}_{2n})$. 
  We first investigate  the weak convergence of the vector
  \begin{align}\label{hankel_fest}
        H_{n, k} = (D_{n, 2}(\vec{m}_{2n}), \ldots, D_{n, 2k}(\vec{m}_{2n}))^T
    \end{align}
    for a fixed $k \in \en$.
    \begin{thm} \label{thmfixk} 
        If $\vec{m}_{2n} \sim \mathcal{U}(\mathcal{M}_{2n})$, then the random vector
        \begin{align*}
         \sqrt{4n}(    H_{n, k} - H_{n, k}^0) \xrightarrow{d} \mathcal{N}(0, \Sigma_k)~,
                   \end{align*}
        where $H_{n, k}^0 = (D_{n, 2}^0, \ldots, D_{n, 2k}^0)^T$ and $D_{n,2k} ^0 $ denotes the
log-determinant  of  the Hankel  matrix corresponding to  the arcsine distribution, that is
  \begin{align} \label{arcdet}
D_{n,2k}^0 = \log \det \Big({  2 (i+j)\choose i+j   }      2^{-2 (i+j)}  \Big)_{i,j=0,\ldots , k} 
~=~ - { k (2 k  +1)}  \log (2) ,
  \end{align}
 and    the asymptotic covariance matrix is given by
        \begin{align} \label{sigma}
            \Sigma_k = (i \wedge j)_{i, j = 1}^k~.
        \end{align}
    \end{thm}
    \begin{proof} 
In all proofs of this paper  we do not reflect the dependence of the canonical moments on
  the vector of random moments and use the notation $ p_{2n, i}  =    p_{2n, i} (\vec{m}_{2n}) $.
    According to Lemma \ref{betadist}  the canonical
  moments $    p_{2n, i} $  are independent and $\beta(2n - i + 1, 2n - i + 1)$ distributed $(i=1,2,\dots,2n)$.  
   Since $q_{2n, i} = 1 - p_{2n, i} \sim \beta(2n - i + 1, 2n - i + 1)$ it follows from Lemma \ref{conv_beta} and the Delta method  that 
            \begin{align}
  \label{can1}          \sqrt{4n} (\log(q_{2n, k}) - \log(\tfrac{1}{2})) &\xrightarrow{d} \mathcal{N}(0, 1) \\
   \label{can2}          \sqrt{4n} (\log(p_{2n, k} (1 - p_{2n, k})) - \log(\tfrac{1}{4})) &\xrightarrow{P} 0
        \end{align}
\par  \medskip
 Next, note that the representation \eqref{arcdet} follows from \eqref{detbound} and the fact that the canonical moments
 of the arcsine distribution are all given by $ 1/2$.
  Consequently, we can decompose the vector
 $H_{n, k}$  as follows 
        \begin{align*}
             \sqrt{4n}(    H_{n, k} - H_{n, k}^0)   = S_n - T_n~,
        \end{align*}
        where the components of the vectors $S_n$ and $  T_n$ are given by 
        \begin{align*}
            S_{ni} &= \sum \limits_{j = 1}^{i} \sqrt{4n}(i - j + 1)(\log(q_{2n, 2j}p_{2n, 2j}) + \log(q_{2n, 2j - 1} p_{2n, 2j - 1}) - 2 \log(\tfrac{1}{4})) ~,\\
            T_{ni} &= \sqrt{4n}\sum \limits_{j = 1}^{i} (\log(q_{2n, 2j}) - \log(\tfrac{1}{2})) ~,
        \end{align*}
        respectively ($i=1,\ldots , k$). 
     Observing  \eqref{can2}  we see  that $S_n$ converges in probability to $0$. The weak convergence $T_n \xrightarrow{d} \mathcal{N}(0, \Sigma_k)$ is a routine exercise that follows from (\ref{can1}) and the independence of the $q_{2n, i}$ $(i = 1, \ldots, 2n)$.
    \end{proof}

  While Theorem \ref{thmfixk}  holds for any fixed $k \in \en$, the following result provides a process version.

       \begin{thm} \label{thm1}
       Let   $\vec{m}_{2n}$ denote a  uniformly distributed  random vector on $\mathcal{M}_{2n} ([0,1])$, then
  $$
  \big\{ \mathcal{G}_n(t)  \big \}_{t\in [0,1]}~:=~ \frac {2}{\sqrt{n}}
   \Big\{  D_{n,2\lfloor  nt \rfloor}   (\vec{m}_{2n}) -   D_{n,2\lfloor  nt \rfloor}^0  +       
     \frac{n}{2} r(t)  \Big\}_{t\in [0,1]} ~\Longrightarrow ~ \{ \mathcal{G}^{[0,1]} (t)\}_{t \in [0,1]}~,
   $$
where $ D_{n,2\lfloor  nt \rfloor}   (\vec{m}_{2n}) $ is defined in \eqref{hanproc},
    \begin{align} \label{rt}
        r(t) &= t + (1 - t)\log(1-t)~
    \end{align}
and $\mathcal{G}^{[0,1]}$ is a centered continuous Gaussian process on the interval $[0, 1]$ with covariance kernel
   \begin{eqnarray}     \label{k1}
            f(s, t) &= \int \limits_0^{s \wedge t} \frac{(t - x)(s - x)}{(1 - x)^2} \; dx = (s \wedge t)(2 - s \vee t) - (s + t - 2)\log(1 - s \wedge t) \label{def_f}
  \end{eqnarray}
    \end{thm}

  {\bf Proof:} It is shown later (more precisely, in the proof of (\ref{ass1})) that the kernel $f$ is in fact nonnegative definite, that is for all $ k\in \en$, $s_1,\ldots, s_k, t_1,\ldots, t_k \in [0, 1]$
the matrices $(f(s_i,t_j))_{i,j=1}^k$ are nonnegative definite.
       A simple calculation shows   that    $   \mathbb{E}[(\mathcal{G}^{[0,1]} (t) - \mathcal{G}^{[0,1]}(s))^4] \le 48 (t - s)^2$,
    and consequently   the existence of the process $\mathcal{G}^{[0,1]} = \{\mathcal{G}^{[0,1]}(t) \}_{t\in [0,1]}$  follows from   Theorem 3.23 in \cite{kall2002}.
    Moreover,  since $\mathcal{G}^{[0,1]}$ is continuous and $l^\infty([0, 1])$ is a complete space, Theorem 1.3.2 in \cite{vaartwell1995}
    shows that $\mathcal{G}^{[0,1]}$ is tight. For  the following discussion we define  
    \begin{align*}
   \tilde \xi_{n, i} &= \log(q_{2n, 2i}p_{2n, 2i}) + \log(q_{2n, 2i - 1}p_{2n, 2i - 1}),  \\
       \xi_{n, i}(t) &= \frac{2}{\sqrt{n}} (\lfloor nt \rfloor - i + 1) \big(\tilde\xi_{n, i} - \mathbb{E}[\tilde\xi_{n, i}]\big),
    \end{align*}
and obtain by \eqref{detbound} 
 the decomposition
\be
\mathcal{G}_n(t) &= S_n(t) + 2R_n(t) - 2T_n(t) + 2U_n(t),
\ee
 where the processes  $S_n, R_n, T_n$ and $U_n$ are defined by
        \begin{align*}
           S_n(t) &= \sum \limits_{i = 1}^{\lfloor nt \rfloor - 1} \xi_{n, i}(t),  \\
            T_n(t) &= \frac{1}{\sqrt{n}}\Big\{ \lfloor nt \rfloor \log(2) - \log(p_{2n, 2\lfloor nt\rfloor}) + \sum \limits_{i = 1}^{\lfloor nt \rfloor - 1}\log(q_{2n,2i})\Big\}, \\
            R_n(t) &= \frac{I\{nt \ge 1\}}{\sqrt{n}} \log(q_{2n,2\lfloor nt \rfloor - 1}p_{2n,2\lfloor nt \rfloor - 1}), \\
            U_n(t) &= \frac{1}{\sqrt{n}}\Big\{  \sum \limits_{i = 1}^{\lfloor nt \rfloor - 1}(\lfloor nt \rfloor - i + 1)\mathbb{E}[\tilde\xi_{n, i}] - D^0_{n, 2\lfloor nt \rfloor} + \frac{n}{2}r(t) + \lfloor nt \rfloor\log(2) \Big\},
        \end{align*}
        respectively.    With these notations  the proof of Theorem \ref{thm1} follows from the assertions
        \begin{align}
 & S_n \Longrightarrow \mathcal{G}^{[0,1]} ~, \label{ass1}  \\
& T_n \Longrightarrow 0  \label{ass2}  ~, \\
& R_n \Longrightarrow 0  \label{ass3} ~,  \\
& ||U_n||_\infty \xrightarrow{n \to \infty} 0 ~, \label{ass4}
    \end{align}
and a simple application of Slutsky's theorem.

 {\it Proof of (\ref{ass1}).}     For each $k \in \mathbb{N}$  consider $0 = t_0 \le t_1 \le \ldots \le t_k \le 1$ and define the 
 $k$-dimensional  random variable
$
            S_n^* := (S_n(t_1), \ldots, S_n(t_k))^T.
$   Let $c = (c_1, \ldots, c_k)^T \in \mathbb{R}^k$ be an arbitrary vector. then
        \begin{align*}
            c^T S_n^* &=\sum \limits_{i = 1}^k c_i S_n(t_i)
            = \sum \limits_{i = 1}^k c_i \sum \limits_{j = 1}^{\lfloor nt_i \rfloor - 1} \xi_{n, j}(t_i)
            = \sum \limits_{i = 1}^k c_i \sum \limits_{l = 1}^i \sum \limits_{j = \lfloor nt_{l - 1} \rfloor \vee 1}^{\lfloor nt_l \rfloor - 1} \xi_{n, j}(t_i) \\
            &= \frac{2}{\sqrt{n}}\sum \limits_{l = 1}^k \sum \limits_{j = \lfloor nt_{l - 1} \rfloor \vee 1}^{\lfloor nt_l \rfloor - 1} \big(\tilde \xi_{n, j} - \mathbb{E}[\tilde \xi_{n, j}]\big)\sum \limits_{i = l}^k c_i (\lfloor nt_i\rfloor -j + 1)
        \end{align*}
In order to calculate the variance  of $c^TS_n^*$ we assume $0 \leq s \leq t \leq 1$, use the approximation \eqref{var_genau} in the Appendix and obtain 
        \begin{align*}
            \operatorname{cov}(S_n(s), S_n(t)) =& \frac{4}{n} \sum \limits_{i = 1}^{\lfloor ns \rfloor - 1}(\lfloor nt \rfloor - i + 1)(\lfloor ns \rfloor - i + 1) \operatorname{Var}(\tilde \xi_{n, i}) \\
            =& \frac{4}{n} \sum \limits_{i = 1}^{\lfloor ns \rfloor - 1}\frac{(\lfloor nt \rfloor - i + 1)(\lfloor ns \rfloor - i + 1)}{4(n - i + 1)^2} \\
            &+ \frac{4}{n} \sum \limits_{i = 1}^{\lfloor ns \rfloor - 1}(\lfloor nt \rfloor - i + 1)(\lfloor ns \rfloor - i + 1) O\left((n - i + 1)^{-3}\right) \\
            =& \frac{1}{n}\sum \limits_{i = 1}^{\lfloor ns \rfloor - 1}\frac{\left(t - \frac{i - 1 + nt - \lfloor nt \rfloor}{n}\right)\left(s - \frac{i - 1 + ns - \lfloor ns \rfloor }{n}\right)}{(1 - \frac{i - 1}{n})^2} + \frac{4}{n} \sum \limits_{i = 1}^{\lfloor ns \rfloor - 1} O((n - i + 1)^{-1}).
        \end{align*}
         Interpreting the first term as Riemann-sum, we can calculate the limit
        \begin{align*}
     \lim_{n\to \infty}   \operatorname{cov}(S_n(s), S_n(t)) =       \int \limits_0^s \frac{(t - x)(s - x)}{(1 - x)^2} \; dx = f(s, t),
        \end{align*}
        which gives
        \begin{align*}
             \lim_{n\to \infty}     \operatorname{Var}(c^T S_n^*) =     \lim_{n\to \infty}   c^T \operatorname{cov}(S_n^*, S_n^*) c \to c^T \Sigma c ~,
        \end{align*}
        where the matrix $\Sigma$ is given by $\Sigma=(f(t_i,t_j))_{i,j=1,\dots,k}$ and  the covariance kernel $f$ is defined in  (\ref{def_f}).
        Consequently we obtain  that this  kernel is nonnegative definite. \\
We now prove the weak convergence of $c^TS_n^*$ by verifying the  Lyapunov--condition. For this purpose we use the notation  $c^* := \max\{|c_1|, \ldots, |c_n|\}$
and obtain
        \begin{align*}
            &\frac{2^4}{n^2}\sum \limits_{l = 1}^k \sum \limits_{j = \lfloor nt_{l - 1} \rfloor \vee 1}^{\lfloor nt_l \rfloor - 1} \mathbb{E}\big[(\tilde \xi_{n, j} - \mathbb{E}[\tilde \xi_{n, j}])^4\big] \Big(\sum \limits_{i = l}^k c_i (\lfloor nt_i\rfloor -j + 1) \Big)^4 \\
            & \le \frac{2^4}{n^2} \sum \limits_{l = 1}^k \sum \limits_{j = \lfloor nt_{l - 1} \rfloor \vee 1}^{\lfloor nt_l \rfloor - 1} \mathbb{E}\big[(\tilde \xi_{n, j} - \mathbb{E}[\tilde \xi_{n, j}]^4\big] (kc^*)^4 (n - j + 1)^4
           ~ \le ~\frac{(2kc^*)^4 C^2}{n} \to 0~,
        \end{align*}
where we have used the estimate  (\ref{viertes_moment}) in Appendix  \ref{appendix}  for  the moments. Consequently,  Lyapunov`s Theorem implies convergence of the finite dimensional distributions, that is
  \begin{align*}
            S_n^* = (S_n(t_1), \ldots, S_n(t_k))^T \xrightarrow{d} \mathcal{N}(0, \Sigma).
        \end{align*}
We finally prove that  $S_n$ is asymptotically tight, that is
      \begin{align} \label{tight}
     \lim_{m\to \infty}        \limsup \limits_{n \to \infty} \mathbb{P}\left(\omega_n\left(\tfrac{1}{m}\right) > \epsilon \right) ~=~0~,
        \end{align}
    where
$
            \omega_n(a) = \sup \left \{ | S_n(t) - S_n(s) | \; | \; 0 \le t - s \le a\right \}
$
        denotes the modulus of continuity of the process  $S_n$.
The statement   \eqref{ass1} then follows  from  Theorem 1.5.4 in \cite{vaartwell1995}.  For a
proof of \eqref{tight}  we introduce  the notation
        \begin{align*}
            d_{n, i} = \frac{2}{\sqrt{n}}
            \begin{cases}
                \lfloor nt \rfloor - \lfloor ns \rfloor & i \le \lfloor ns \rfloor - 1 \\
                \lfloor nt \rfloor - i + 1 & \lfloor ns \rfloor - 1 < i \le \lfloor nt \rfloor - 1 \\
                0 & \text{else}
            \end{cases}~,
        \end{align*}
and obtain the following representation
        \begin{align*}
            S_n(t) - S_n(s) = \sum \limits_{i = 1}^n d_{n, i} \big(\tilde \xi_{n, i} - \mathbb{E}[\tilde \xi_{n, i}]\big).
        \end{align*}
        The inequalities (\ref{var_grob}) and (\ref{viertes_moment}) in the Appendix then yield 
        \begin{align}
            \mathbb{E}\left[\left(S_n(t) - S_n(s)\right)^4\right]
           & \le
             C^2 \Big(\sum \limits_{i = 1}^n d_{n, i}^2 \frac{1}{(n - i + 1)^2}  \Big)^2 \notag \\
           & \le (2C)^2   \Big(\frac{1}{n}  \sum \limits_{i = 1}^{\lfloor ns \rfloor - 1} \frac{(t - s + \frac{1}{n})^2}{(1 - \frac{i}{n} + \frac{1}{n})^2} + 
           \frac{1}{n}
           \sum \limits_{i = \lfloor ns \rfloor \vee 1}^{\lfloor nt \rfloor - 1}  \frac{(t -\frac{i - 1}{n})^2}{(1 - \frac{i - 1}{n})^2}\Big)^2 \notag \\
           & \le (2C)^2 \Big(I\left\{s \ge \tfrac{2}{n}\right\} \left(t - s + \tfrac{1}{n}\right)^2 \int \limits_{\frac{1}{n}}^s \frac{1}{\left(1-x + \frac{1}{n}\right)^2} \; dx + \frac{(\lfloor nt \rfloor - \lfloor ns \rfloor)}{n} \Big)^2 \notag \\
           & \le (2C)^2\Big(I\left\{s \ge \tfrac{2}{n}\right\}\left(t - s + \tfrac{1}{n}\right)^2 \left(\tfrac{1}{1 - s +{1}/{n}} - 1 \right) + \left(t - s + \tfrac{1}{n}\right) \Big)^2\notag \\
           & \le (4C)^2 \big(t - s + \tfrac{1}{n}\big)^2 .\label{ungl_mom_modul}
        \end{align}
        Consequently, we obtain
        \begin{align} \label{term1}
            &\limsup \limits_{n \to \infty}\sum \limits_{k = 1}^m E\big [( S_n (\tfrac{k}{m} ) - S_n (\tfrac{k - 1}{m} ))^4\big]
            \le \limsup \limits_{n \to \infty}\sum \limits_{k = 1}^m (4C)^2\big (\tfrac{1}{m} + \tfrac{1}{n}\big)^2 = \frac{(4C)^2}{m}  .
        \end{align}
Now assume that $ 0 \le r \le s \le t \le 1 $. If  $t - r \ge \tfrac{1}{n}$, H\"{o}lders's inequality and (\ref{ungl_mom_modul}) yield
the estimate
        \begin{align*}
            \mathbb{E}\left[ (S_n(s)-S_n(r))^2 (S_n(t)-S_n(s))^2\right]
            &
            \le (4C)^2\left(s - r + \tfrac{1}{n}\right)\left(t - s + \tfrac{1}{n} \right) \\
            & \le (4C)^2 \left(\tfrac{t - r}{2} + \tfrac{1}{n}\right)^2 \le (6C)^2(t - r)^2,
        \end{align*}
        which also holds if  $t - r < \tfrac{1}{n}$ (because we  have $S_n(r) = S_n(s)$ or $S_n(s) = S_n(t)$ in this case). Therefore
Lemma 3.1 in \cite{schowell1986}  and \eqref{term1} show
  that
             \begin{align*}
         \limsup \limits_{n \to \infty}
         \mathbb{P}\left(\omega_n\left(\tfrac{1}{m}\right) \ge \epsilon\right) & \le
                 \limsup \limits_{n \to \infty}    \frac{1}{\epsilon^4} \Big\{  \sum \limits_{k = 1}^m
         \mathbb{E}\big [\big (S_n (\tfrac{k}{m}) -S_n (\tfrac{k - 1}{m} ) \big)^4 \big]+ \frac{K(6C)^2 }{ m }  \Big\} \\
         & = \frac{(4C)^2 + K(6C)^2 }{\epsilon^4 m}
      \end{align*}
      for an absolute constant $K$. This proves \eqref{tight} and completes the proof of \eqref{ass1}.  \medskip

      {\it Proof of \eqref{ass2}  and  \eqref{ass3}:} These statements follow by similar arguments as given in the proof of assertion \eqref{ass1}
using the estimates (\ref{erwartung}) - (\ref{var_trigamma}) in Appendix  \ref{appendix}. The details are omitted for the sake of brevity.
\medskip

      {\it Proof of \eqref{ass4}:} By  \eqref{arcdet}  we have 
        \begin{align*}
            D^0_{n, 2\lfloor nt \rfloor}  &= - (2\lfloor nt \rfloor +1)  \lfloor nt \rfloor  \log(2)~,
        \end{align*}
    and  the estimate (\ref{erwartung}) from Appendix \ref{appendix} yields the approximation
        \begin{align*}
            \sum \limits_{i = 1}^{\lfloor nt \rfloor - 1} (\lfloor nt \rfloor - i + 1) \mathbb{E}[\tilde \xi_{n, i}]
            &=\sum \limits_{i = 1}^{\lfloor nt \rfloor - 1} (\lfloor nt \rfloor - i + 1) \big(-4\log(2) - \tfrac{1}{2(n - i + 1)} + O\big(\tfrac{1}{(2n - 2i + 1)^2}\big)\big) \\
            &= -2 \log (2) \big( \lfloor nt \rfloor^2  + 2\lfloor nt \rfloor \big)+ \log(16) - \tfrac{\lfloor nt \rfloor}{2} + \tfrac{n -\lfloor nt \rfloor }{2}(G_n - G_{n - \lfloor nt \rfloor + 1}) \\
            & + O(\log(n))
        \end{align*}
       (uniformly with respect to $t \in [0,1]$), where   $G_n =\sum_{i=1}^n \tfrac{1}{i} $ is the $n$th partial sum of the harmonic series. Therefore
        \begin{align*}
            U_n(t) &= \frac{1}{\sqrt{n}} \Big (\tfrac{nt-\lfloor nt \rfloor}{2} +\tfrac{n}{2}(1 - t)\log(1 - t) +\tfrac{n - \lfloor nt \rfloor }{2}(G_n - G_{n - \lfloor nt \rfloor + 1})\Big) + O\Big (\frac{\log(n)}{\sqrt{n}}\Big).
        \end{align*}
        Using the approximation $G_n = \log(n) + \gamma + O(\tfrac{1}{n})$, where $\gamma$ is the Euler-Mascheroni constant, we can easily see that
        \begin{align*}
            U_n (t) &= \frac{1}{\sqrt{n}} \Big \{ \tfrac{n}{2}(1 - t)\log(1 - t) - \tfrac{n - nt}{2}\log\Big (1 - t + \tfrac{nt - \lfloor nt \rfloor + 1}{n}\Big)\Big \} 
            + O\Big(\frac{\log(n)}{\sqrt{n}}\Big) \\
            &= \frac{-n(1 - t)}{2\sqrt{n}}\log\Big (1 + \tfrac{nt - \lfloor nt \rfloor + 1}{n(1 - t)}\Big ) + O\Big(\frac{\log(n)}{\sqrt{n}}\Big)   ~=~ o(1)~,
        \end{align*}
        uniformly with respect $t\in [0,1]$, which completes the proof of Theorem \ref{thm1}.  \hfill $\Box$

\bigskip

\begin{rem} {\rm Similar results as stated in Theorem \ref{thmfixk} and \ref{thm1} 
can be obtained for the Hankel  matrices  $\underline{H}_{2n+1} = (m_{i+j+1} )_{i,j=0}^n$,
$\overline{H}_{2n} = (m_{i+j+1}- m_{i+j +2})_{i,j=0}^{n-1}$ and $\overline{H}_{2n+1} = (m_{i+j} -m_{i+j +1})_{i,j=0}^n$, which are
commonly used to characterize Hausdorff moment sequences [see \cite{karlin1966}]. The details are omitted for the sake of brevity.
}
\end{rem}

 \subsection{Hankel determinants from $\mathcal{M}_{2n} ([0,\infty ))$ and  $\mathcal{M}_{2n} (\er )$ } \label{sec32}

In this section we will derive analogues of Theorem \ref{thm1} for random moment sequences on  unbounded moment spaces, where the corresponding distributions are defined by \eqref{momdichte0infty} and \eqref{momdichteer}, respectively.
For the sake of brevity we omit the discussion of $D_{n,2k}  (\vec{m}_{2n}) $ for fixed $k$ (corresponding results can
be easily obtained using similar arguments as given in the proof of Theorem \ref{thmfixk}) and concentrate on the stochastic process
$ \{ D_{n,2\lfloor  nt \rfloor}   (\vec{m}_{2n}) \}_{t\in [0,1]}$.

       \begin{thm} \label{thm2}
       Let   $\vec{m}_{2n}$ denote a  random vector on $\mathcal{M}_{2n} ([0,\infty))$ with density $g_n^{(\gamma,\delta)}$ defined in \eqref{momdichte0infty}, where
       $\gamma_{2n,1}, \ldots  \gamma_{2n,2n} > 0$  are bounded by a constant which does not depend on $n$ and $\delta_{2n, i} = 2n - i + 1 + \gamma_{2n, i}$, then
  $$
  \big\{ \mathcal{G}_n(t)  \big \}_{t\in [0,1]}~:=~
   \Big\{ {1\over n}  D_{n,2\lfloor  nt \rfloor}   (\vec{m}_{2n})  \Big\}_{t\in [0,1]} ~\Longrightarrow ~ \big \{ \mathcal{G}^{[0,\infty)} (t) \}_{t\in [0,1]} 
   $$
where $ D_{n,2\lfloor  nt \rfloor}   (\vec{m}_{2n}) $ is defined in \eqref{hanproc},
and $\mathcal{G}^{[0,\infty)} $ is a continuous Gaussian process on the interval $[0, 1]$ with mean
$ -  r(t) /2 $
and covariance kernel
   \begin{align}  \label{def_g}    
            g(s, t) &= \int \limits_0^{s \wedge t} \frac{(t - x)(s - x)}{1 - x} \; dx \\
            &= \frac {1}{2}(s \wedge t)(s+t-2+ s \vee t) +  (s-1)(t-1)  \log(1 - s \wedge t) \notag
  \end{align}
    \end{thm}

  {\bf Proof:}     We will use the decomposition
 ${1\over n}   D_{n,2\lfloor  nt \rfloor}   (\vec{m}_{2n}) = S_n^{[0, \infty)}(t) + R_n^{[0, \infty)}(t)$,
        where the processes $ S_n^{[0, \infty)} $ and $R_n^{[0, \infty)}$  are defined by
        \begin{align*}
            S_n^{[0, \infty)}(t) &=  {1\over n}   \left( D_{n,2\lfloor  nt \rfloor}   (\vec{m}_{2n}) - \mathbb{E}[ D_{n,2\lfloor  nt \rfloor}   (\vec{m}_{2n}) ] \right) \\
            R_n^{[0, \infty)}(t) &= \frac{1}{n} \mathbb{E}[ D_{n,2\lfloor  nt \rfloor}   (\vec{m}_{2n}) ]
        \end{align*}
        Observing the fact that $bX \sim \gamma(a, 1)$, whenever  $X \sim \gamma(a, b)$, and using the approximations (\ref{mean_gamma}) - (\ref{fourth_moment_gamma}) from the Appendix
        it can be shown by similar arguments as given in the proof of Theorem \ref{thm1}  that  $S_n^{[0, \infty)}$ converges weakly to a centered continuous Gaussian process on the interval $[0, 1]$
        with covariance kernel \eqref{def_g}.  For the remaining term $R_n^{[0, \infty)}(t)$ we use  \eqref{detunbound}, Lemma \ref{gammadist}
   and      the approximation
                \begin{align*}
            \mathbb{E}(\log(z_{2n, i})) &= \mathbb{E}\left[\log \left(z_{2n, i}\cdot(2n - i + 1 + \gamma_{2n, i})\right) \right] - \log(2n - i + 1 + \gamma_{2n, i}) \\
            &= -\frac{1}{2(2n - i + 1 + \gamma_{2n, i})} +O((2n - i + 1 + \gamma_{2n, i})^{-2}).
        \end{align*}
This yields  (uniformly with respect to $t \in [0,1]$)
        \begin{align*}
            R_n^{[0, \infty)}(t) = -\frac{1}{4n}\sum \limits_{i = 1}^{\lfloor nt \rfloor}\left( \frac{t - \frac{i}{n}+ \frac{\lfloor nt \rfloor - nt + 1}{n}}{1 - \frac{i}{n} + \frac{\gamma_{2n, 2i} + 1/2}{n}} + \frac{t - \frac{i}{n}+ \frac{\lfloor nt \rfloor - nt + 1}{n}}{1 - \frac{i}{n} + \frac{\gamma_{2n, 2i - 1} + 1}{n}} \right) + O\left(\frac{\log(n)}{n}\right),
        \end{align*}
        and a careful calculation shows that this term converges uniformly to $-r(t)/2$, where $r(t)$ is defined in \eqref{rt}.
This yields the assertion. \hfill  $\Box$

\bigskip

We conclude this section with a corresponding result  for the moment space  $\mathcal{M}_{2n} (\er)$. The proof is similar to that of Theorem \ref{thm1} and therefore
omitted. 
    \begin{thm} \label{thm3}
       Let   $\vec{m}_{2n-1}$ denote a  random vector on $\mathcal{M}_{2n-1}  (\er)$ with density $h_n^{(\gamma,\delta)}$ defined in \eqref{momdichte0infty}, where
       $\gamma_{n,1}, \ldots  \gamma_{n,n}$  are bounded by a constant which does not depend on $n$ and $\delta_{n, 2i} = 2n - 2i + \gamma_{n, i}$, then
  $$
   \Big\{ {1\over n}  D_{n,2\lfloor  (n-1)t \rfloor}   (\vec{m}_{2n-1})  \Big\}_{t\in [0,1]} ~\Longrightarrow ~ \{\mathcal{G}^{\mathbb{R}}(t)\}_{t \in [0,1]}
   $$
where $ D_{n,2\lfloor  (n-1) t \rfloor}   (\vec{m}_{2n}) $ is defined in \eqref{hanproc},
and $\mathcal{G}^{\mathbb{R}} $ is a continuous Gaussian process on the interval $[0, 1]$ with mean
$ -  r(t) /4 $  and covariance kernel $ g(s, t)/2$, defined by \eqref{rt} and  \eqref{k1}, respectively.
    \end{thm}
    

\section{Large deviations} \label{sec4}
 \def\theequation{4.\arabic{equation}}
 \setcounter{equation}{0}
 Throughout this section we consider large deviation principles (LDP) for 
 the moment space $\mathcal{M}_{2n}([0,1])$. Similar results can be obtained for moment spaces corresponding to unbounded intervals.
    For fixed $k$ the sequence $(H_n^k)_{n\in \en} $ defined in (\ref{hankel_fest}) for 
    a uniformly distributed vector  $ \vec{m}_{2n}$ on the moment space $\mathcal{M}_{2n}([0,1])$   satisfies an LDP 
     with a good rate function. To see this, observe that the sequence of canonical moments $(Y_n)_{n\in \en}  = ((p_{2n, 1}, \ldots, p_{2n, 2k})^T)_{n\in \en} $ satisfies a large deviation principle with good rate function
    \begin{align*}
        I(x) = 2\sum \limits_{i = 1}^{2k}\left( -\log(x_i - x_i^2) - \log(4)\right)
    \end{align*}
    (c.f. \cite{gamloz2004}). As the  function that maps the canonical moments to the logarithms of the Hankel-determinants is obviously continuous,  the contraction principle [Theorem 4.2.1 in \cite{demzeit1998})] shows that $(H_n^k)_{n\in \en} $ satisfies an LDP with a good rate function. However, due to the complicated form of this map it is not possible to explicitly represent the corresponding 
     rate function in terms of standard functions.  \\
 The investigation of  LDP-properties of  the logarithm of the lower Hankel determinant
    with increasing dimension turns out to be substantially more complicated,  and we consider again the
     process $\{  D_{n,2\lfloor  nt \rfloor}   (\vec{m}_{2n})  \}_{t\in [0,1]}$, which 
    has to be normalized differently, that is 
    \begin{align*}
        Z_n(t)  &= -{1\over n}  \big (  D_{n,2\lfloor  nt \rfloor}   (\vec{m}_{2n})   -D_{n, 2\lfloor nt \rfloor}^0 \big )~, 
    \end{align*}
where $q_{n, 0} =1$. 
Let  $\mathcal{S}([0, 1])$ denote the space of all signed regular Borel measures on the interval $[0, 1]$ endowed  
 with the weak-$*$-topology (with ($C([0, 1]), || \cdot||_\infty)$ as predual). Then its (topological) dual space is the space $C([0, 1])$ of all continuous functions on the interval $[0, 1]$. In the following  we interpret the process $Z_n$ as the distribution function of a random measure $\nu_n \in \mathcal{S}([0, 1])$.  To be precise note that the  process   $Z_n$ 
 is piecewise constant with jumps at the points  $ \tfrac{1}{n}, \ldots, \tfrac{n}{n}$. Therefore $\nu_n$ is a linear combination 
 of Dirac-measures and a simple calculation shows that 
        \begin{align} \label{nun} 
            \nu_n = -\frac{1}{n} \sum \limits_{i = 1}^n \Big \{ \sum \limits_{j = 1}^i \log(4q_{2n, 2j - 1}p_{2n, 2j - 1}) + \sum \limits_{j = 1}^{i - 1} \log(4q_{2n, 2j}p_{2n, 2j}) + \log(2p_{2n, 2i})\Big\} \delta_{\frac{i}{n}}~, 
        \end{align}
 where $\delta_t$ denotes the Dirac measure at the point $t \in [0,1]$.
In order to investigate the large deviation properties of the sequence of random measures $\{ \nu_n\}_{n \in \en}$ we first
derive the limit of the (normalized) logarithmic moment generating function.

    \begin{thm}  \label{lamdarst}
    Let $\nu_n$ denote the random measure defined in (\ref{nun}). 
    For any Riemann-integrable function $f \in l^\infty([0, 1])$ we have 
        \begin{align*}
            \Lambda (f) = \lim \limits_{n \to \infty} \frac{1}{n}\log \mathbb{E}[e^{n \nu_n(f)}] =
            \begin{cases}
                -\int \limits_0^1 \log \big(1 - \frac{G(x)}{2(1-x)}\big) \; dx & K < 2 \\
                \infty & K > 2
            \end{cases}~,
        \end{align*}
where 
    \begin{align*}
        G(x) = \int \limits_x^1 f(t) \; dt ~; \qquad K = \sup \limits_{x \in [0, 1)} \frac{G(x)}{1 - x} ~. 
    \end{align*}
        It is in general unknown what happens in the case $K = 2$.
    \end{thm}
    \begin{proof}
        Interpreting the sequences
        \begin{align*}
            x_{n, j} = -\sum \limits_{i = j}^n f\left(\tfrac{i}{n}\right) \text{ and } y_{n, j} = -\sum \limits_{i = j + 1}^n f\left(\tfrac{i}{n}\right)
        \end{align*}
        as Riemann-sums, we get the approximations
        \begin{align}
            \sup \limits_{j = 1, \ldots, n}\left|G\left(\tfrac{j}{n}\right) + \tfrac{x_{n, j}}{n}\right| \xrightarrow{n \to \infty} 0\label{xj_riemann}~, 
            \sup \limits_{j = 1, \ldots, n}\left|G\left(\tfrac{j}{n}\right) + \tfrac{y_{n, j}}{n}\right| \xrightarrow{n \to \infty} 0 ~.
        \end{align}
        This yields for the logarithm of the moment generating function 
        \begin{align*}
            &\frac{1}{n} \log\left(\mathbb{E}[\exp(n\nu_n(f))]\right) \\
            =& \frac{1}{n} \sum \limits_{j = 1}^n \left(\log\left(\mathbb{E}\left[q_{2n, 2j - 1}^{x_{n, j}}p_{2n, 2j - 1}^{x_{n, j}}\right]\right) + \log\left(\mathbb{E}\left[q_{2n, 2j}^{y_{n, j}}p_{2n, 2j}^{x_{n, j}}\right]\right) + (3x_{n, j} + y_{n, j})\log(2)\right) .
        \end{align*}
     For the determination of the limit we now consider the two  cases $K > 2$ and $K < 2$ separately.  \\
\smallskip
        {\bf (1) } In the case $K > 2$ we choose constants $\delta, C > 0$ such that for all sufficiently large $n$ there exists a $j_n \in \{1, \ldots, n\}$ with $1 - \frac{j_n}{n} > C$ and $G(\tfrac{j_n}{n}) \ge (2 + \delta)(1- \tfrac{j_n}{n})$ (this is possible since the function 
        $G$ is continuous). Choosing another constant $0 < \epsilon < \delta C$ and considering (\ref{xj_riemann}), we get the following approximation for sufficiently large $n$:
        \begin{align*}
            \tfrac{x_{n, j_n}}{n} \le \left|\tfrac{x_{n, j_n}}{n} + G\left(\tfrac{j}{n}\right)\right| - G\left(\tfrac{j_n}{n}\right) \le \epsilon - (2 + \delta)(1 - \tfrac{j_n}{n})~.
        \end{align*}
        Therefore $
            2n - 2j_n + 1 + x_{n, j_n} \le 1 + n(\epsilon - \delta C) < -1,
$
        which yields $\mathbb{E}[q_{n, 2j_n - 1}^{x_{n, j_n}}p_{n, 2j_n - 1}^{x_{n, j_n}}] = \infty$ and the assertion follows. \\
\smallskip
       {\bf (2)}  In the case $K < 2$ we use the formula
        \begin{align*}
            \log (\Gamma(x)) = (x - \tfrac{1}{2}) \log(x) - x + \frac{\log(2\pi)}{2} + 2 \phi_0(x) ~, 
                    \end{align*}
where
             \begin{align} \label{phi0} 
        \phi_0(x) = \int \limits_0^\infty \frac{\arctan \left(\frac{t}{x}\right)}{\exp(2\pi t) - 1} \; dt
        \end{align}
        [cf. (4.3) in \cite{detgam2007}]. Using the representation \eqref{nun} we can show that
        \begin{align} \label{decomp1}
            \frac{1}{n}\mathbb{E}[\exp(n \nu_n(f))]  &=  B_{n,1} +B_{n,2} +B_{n,3} + B_{n,4 }  \\
             & + R(2n - 2j + 2, x_{n, j}, x_{n, j}) + R(2n - 2j + 1, x_{n, j}, y_{n, j}) ~, \nonumber 
        \end{align}
        where 
             \begin{align*}
             B_{n,1} &= - \frac{1}{2n} \sum \limits_{j = 1}^n    \log\left(1 + \tfrac{x_{n, j}}{2n - 2j + 2} \right) ~, \\ 
               B_{n,2} &= - \frac{1}{2n} \sum \limits_{j = 1}^n    \log\left(1 + \tfrac{x_{n, j}}{2n - 2j + 1} \right) ~, \\ 
                B_{n,3} &=  \frac{1}{n} \sum \limits_{j = 1}^n  (2n - 2j + 1 + x_{n,j}) \log\left(1 + \tfrac{-f\left({j}/{n}\right)}{2(2n - 2j + 1) + x_{n, j} + y_{n, j}}\right)~,\\
  B_{n,4} &=  \frac{1}{n} \sum \limits_{j = 1}^n 
          (2n - 2j + 1 + y_{n,j}-\tfrac{1}{2}) \log\left(1 + \tfrac{f\left( {j}/{n}\right)}{2(2n - 2j + 1) + x_{n, j} + y_{n, j}}\right)~,
        \end{align*}
and the remaining two terms are defined by 
        \begin{align*}
            R(a, x, y) = 2(\phi_0(a+x) - \phi_0(a) + \phi_0(a+y) - \phi_0(a)) - 4(\phi_0(2a + x + y) - \phi_0(2a))~. 
        \end{align*}
We now investigate  the terms in this decomposition   separately. The first term $   B_{n,1}$ can be interpreted as Riemann-sum, using (\ref{xj_riemann}), that is 
        \begin{align} \label{b1}
      B_{n,1}  = -\frac{1}{2n } \sum \limits_{j = 1}^n  \log 
            \Big(1-  \frac{G\left(\frac{j}{n}\right)+o(1)}{2\left(1 - \frac{j - 1}{n}\right) + o(1)}\Big) 
            \xrightarrow{n \to \infty} -\frac{1}{2} \int \limits_0^1 \log \Big (1 - \frac{G(x)}{2(1-x)}\Big) \; dx~.
        \end{align}
        Analogously, the second term converges to the same limit, i.e. 
          \begin{align} \label{b2}
      B_{n,2}               \xrightarrow{n \to \infty} -\frac{1}{2} \int \limits_0^1 \log \Big (1 - \frac{G(x)}{2(1-x)}\Big) \; dx~.
        \end{align}
For the the terms $  B_{n,3} $ and $   B_{n,4} $ we  use the  Taylor-approximation $\log(1 + x) = x + O(x^2) \; (x \to 0)$ 
and obtain 
        \begin{align}  \label{b3}
     B_{n,3} & =    -\frac{1}{2n}
                   \sum \limits_{j = 1}^n  \frac{2\left(1-\frac{j}{n}\right) - G\left(\frac{j}{n}\right) + o(1)}{2\left(1-\tfrac{j}{n}\right) - G\left(\frac{j}{n}\right) + o(1)} 
         f\left(\tfrac{j}{n}\right)+ O\Big (\frac{1}{n}\Big) \xrightarrow{n \to \infty} -\frac{G(0)}{2} ~, \\
 \label{b4}
     B_{n,4} &= 
            \xrightarrow{n \to \infty} \frac{G(0)}{2}~,
        \end{align}
        and it remains to show that the last two terms in \eqref{decomp1} are asymptotically negligible. \\
         For this purpose we note  that  the following inequality holds for the function $\phi_0$ defined in \eqref{phi0} [cf.  formula (4.10) in \cite{detgam2007}]
        \begin{align*}
            |\phi_0(a + x) - \phi_0(a)| \le C \frac{|x|}{(a \wedge (a + x))^2} \quad \text{with} \quad C = \int \limits_0^\infty \frac{t}{\exp(2\pi t) - 1} \; dt ~,
        \end{align*}
        where $a > 0, x > -a$.   This gives 
        \begin{align*}
            |R(a, x, y)| \le 2C \Big (\frac{|x|}{(a \wedge (a + x))^2} + \frac{|y|}{(a \wedge (a + y))^2}\Big) + 4C \Big (\frac{|x + y|}{(2a \wedge (2a + x + y))^2}\Big)~,
        \end{align*}
        and  using this inequality to estimate the terms $R(2n - 2j + 2, x_{n, j}, x_{n, j}) $ and $ R(2n - 2j + 1, x_{n, j}, y_{n, j}) $ in \eqref{decomp1}
yields six terms, which have a similar form. 
        For the sake of brevity we will only show exemplarily  the convergence
        \begin{align*}
       D_n :=      \frac{1}{n} \sum \limits_{j = 1}^n \frac{|x_{n, j}|}{((2n - 2j + 2) \wedge (2n - 2j + 2+ x_{n, j}))^2} \xrightarrow{n \to \infty} 0~.
        \end{align*}
      The  other five sums can be approximated in a similar way and the details are omitted. For sufficiently small $\epsilon > 0$ and 
      sufficiently large $n$, we obtain by similar arguments as in the case $K > 2$:
        \begin{align*}
            x_{n, j} &\ge -(\epsilon + K)(n - j) - \epsilon j ~,\\
            2n - 2j + x_{n, j} &\ge (2 - K - 2\epsilon)(n - j)~.
        \end{align*}
        Choosing $\delta = \min\{2 - K - 2\epsilon, 2\} > 0$, we get the inequalities
        \begin{align*}
            (2n - 2j + 2) \wedge (2n - 2j + 2 + x_{n, j}) &\ge \delta(n - j + 1) ~,\\
            |x_{n, j}| \le n(|G(\tfrac{j}{n})| + \epsilon) \le ||f||_\infty (n - j + 1) + n \epsilon ~.
        \end{align*}
        This yields
        \begin{align} \label{Dn} 
   \limsup_{n\to \infty}         D_n   ~\le ~
     \limsup_{n\to \infty}  \Big\{ 
     \frac{||f||_\infty}{\delta^2 n} \sum \limits_{j = 1}^n \frac{1}{n - j + 1} + \frac{\epsilon}{\delta^2} \sum \limits_{j = 1}^n \frac{1}{(n - j + 1)^2} \Big\} ~=~
     \frac{\epsilon}{\delta^2} \sum \limits_{j = 1}^\infty \frac{1}{j^2} ~.
        \end{align}
        Considering the limit ${\epsilon \searrow 0} $ on the right hand side of \eqref{Dn} we obtain 
that the last two terms in \eqref{decomp1}  converge to $0$, and the assertion follows from \eqref{b1} - \eqref{b4}.
   \end{proof}
   \begin{rem} {\rm 
        Note that for the application in Lemma \ref{exp_tight} and Proposition \ref{ldp_prozess}, it would be suffivient  to prove the preceeding theorem only for functions $f \in C([0, 1])$. However, we have chosen to prove it more generally for any Riemann-integrable function $f$, as this allos us to apply the formula for the limit in the proof of Theorem \ref{ldp_fixed}.
        }
    \end{rem}

    \begin{lem}\label{exp_tight}
    The sequence 
        $(\nu_n)_{n\in \en}$ of  random measures
        defined by (\ref{nun})  is  exponentially tight.
    \end{lem}
    \begin{proof}
        By the Banach-Alaoglu theorem the set
        \begin{align*}
            K_\alpha = \Big \{\mu \in \mathcal{S}([0, 1]) \; \Big |\; \sup \limits_{\substack{f \in C([0, 1]) \\ ||f|| \le 1}} \mu(f) \le \alpha\Big \}
        \end{align*}
        is compact (note that we endowed $\mathcal{S}([0, 1])$ with the weak-$*$-topology). We define the modified measure
        \begin{align*}
            \nu_n' = -\frac{1}{n} \sum \limits_{i = 1}^n \Big \{ \sum \limits_{j = 1}^i \log(4q_{2n, 2j - 1}p_{2n, 2j - 1}) + \sum \limits_{j = 1}^{i - 1} \log(4q_{2n, 2j}p_{2n, 2j}) + \log(p_{2n, 2i})\Big\} \delta_{\frac{i}{n}} ~.
        \end{align*}
        Observing $\nu_n = \nu_n' - \frac{\log(2)}{n} \sum \limits_{i = 1}^n \delta_{\frac{i}{n}}$ we can see $\nu_n(f) \le \nu_n'(f) + \log(2)$ for all $f \in C([0, 1])$ with $||f||_\infty \le 1$. Since $\nu_n'$ is a positive measure, we get by Markov's inequality
        \begin{align*}
            \frac{1}{n} \log \mathbb{P}(\nu_n \in K_\alpha^c) &\le \frac{1}{n} \log \mathbb{P}\Big (\sup \limits_{\substack{f \in C([0, 1]) \\ ||f|| \le 1}}\nu_n'(f) > \alpha - \log(2) \Big) \le \frac{1}{n} \log \mathbb{E}[\exp(n \nu_n'(1))] - \alpha + \log(2) \\
            &= \frac{1}{n} \log \mathbb{E}[\exp(n (\nu_n(1) + \log(2)) ] - \alpha + \log(2) \xrightarrow{n \to \infty} \Lambda(1) - \alpha + 2\log(2)~,
        \end{align*}
        which yields the assertion.
    \end{proof}

    \begin{prop} \label{ldp_prozess}
        Let $\Lambda^*$ be the Fenchel-Legendre transform of $\Lambda$ and let $E$ denote the set of all exposed points of $\Lambda^*$ which have an exposing hyperplane $\lambda$ that satisfies
        \begin{align*}
            \lim \limits_{n \to \infty} \frac{1}{n} \log \mathbb{E}[\exp(\nu_n(n \lambda)] \text{ exists and } \Lambda(\gamma \lambda) < \infty \text{ for some } \gamma > 1.
        \end{align*}
        Then
        \begin{align*}
            - \inf \limits_{x \in E \cap \Gamma^\circ} \Lambda^*(x) \le \liminf \limits_{n \to \infty} \frac{1}{n} \log \mathbb{P}(\nu_n \in \Gamma) \le \limsup \limits_{n \to \infty} \frac{1}{n} \log \mathbb{P}(\nu_n \in \Gamma)\le -\inf \limits_{x \in \overline{\Gamma}} \Lambda^*(x)
        \end{align*}
        for all measurable sets $\Gamma \subset \mathcal{S}([0, 1])$.
    \end{prop}
    \begin{proof}
        This follows directly from  Baldi's theorem [c.f. Theorem 4.5.20  in \cite{demzeit1998}].
    \end{proof}
   \medskip
   
    The main difficulty in proving an LDP for the process $\{Z_n(t)\}_{t \in [0,1]}$  consists in  the fact that an explicit representation of the Fenchel-Legendre transform $\Lambda^*$
    is not available. This makes it difficult to eliminate the set $E$ in the lower bound in Theorem \ref{ldp_prozess}. On the other hand - in contrast to the LDP for the process  $\{Z_n(t)\}_{t \in [0,1]}$ - the LDP for the random variable $Z_n(t)$ with a fixed $t$ can be 
    established.
    
    \begin{thm}\label{ldp_fixed}
        For a fixed $t \in (0, 1]$ the sequence $(Z_n(t))_{n\in \en} $ satisfies a large deviation principle with good rate function
        \begin{align*}
            \Lambda^*(x) = \sup \limits_{\lambda < \frac{2}{t}}\Big \{\lambda x + \int_0^t \log \left(1 - \tfrac{\lambda(t-y)}{2(1-y)}\right) \; dy\Big\}
        \end{align*}
    \end{thm}
    \begin{proof}
        We will again apply Baldi's theorem. To calculate the normalized cumulant generating function of $Z_n(t)$, note that
        \begin{align*}
            \lambda Z_n(t) = \nu_n(\lambda I\{\cdot \le t\})~,
        \end{align*}
and Theorem \ref{lamdarst}        yields
        \begin{align*}
            \Lambda_t(\lambda) &= \lim \limits_{n \to \infty} \frac{1}{n} \log \mathbb{E}[\exp(n \lambda Z_n(t))]\\
            &  = \Lambda(\lambda I\{ \cdot \le t\}) = \begin{cases}
                - \int_0^t \log \left(1 - \frac{\lambda(t-y)}{2(1-y)}\right) \; dy & \lambda < \frac{2}{t} \\
                \infty & \lambda > \frac{2}{t}
            \end{cases} .
        \end{align*}
 It now follows by similar arguments as given in the proof of Lemma \ref{exp_tight} that the sequence $( Z_n(t)) _{n \in \en} $ 
 is exponentially tight (note that we can use the euclidean topology on $\mathbb{R}$ because the interval 
 $[0, \alpha]$ is compact) and Baldi's theorem yields an analogue of the
  inequality in Theorem \ref{ldp_prozess}, where the set $E$ has
  to be replaced by an analogue set $E_t$. It remains to prove that the lower bound remains correct if one removes the set $E_t$.\

        In order to see this, we define the new function
        \begin{align*}
            \tilde \Lambda_t: 
                \begin{cases}
                \mathbb{R} &\to (-\infty, \infty] \\
            \lambda &\mapsto 
            \begin{cases}
                \Lambda_t(\lambda) & \mbox{ if }  \lambda \ne \frac{2}{t} \\
                \lim \limits_{\epsilon \searrow 0} \Lambda_t(\lambda - \epsilon) & \mbox{ if }  \lambda = \frac{2}{t}
            \end{cases}
            \end{cases}
        \end{align*}
        $\Lambda_t$ and $\tilde \Lambda_t$ have the same Fenchel-Legendre transform and it is therefore sufficient  to prove
        \begin{align} \label{eq_gt}
            \inf \limits_{x \in E_t \cap F} \tilde \Lambda_t^*(x) = \inf \limits_{x \in F} \tilde \Lambda_t^*(x)
        \end{align}
        for all open sets $F \subset \mathbb{R}$. It is easy to see that $\tilde \Lambda_t$ is an essentially smooth function and the identity (\ref{eq_gt}) follows by an adaptation of the arguments in the proof of 
        the Gärtner-Ellis  theorem [Theorem  2.3.6 in \cite{demzeit1998}]. By Lemma 1.2.18 in the same reference 
        the rate function $\Lambda^*$ is a good rate function, which yields the assertion.
    \end{proof}
    \begin{rem} {\rm 
        It is possible also to prove the preceeding theorem more directly by an application of the Gärtner-Ellis theorem, as the limit of $\frac{1}{n} \log \mathbb{E}[e^{n \lambda Z_n(t)}]$ can be calculated using Stirling's approximation. However, these calculations would essentially be a repetition of the calculations done in the proof of Theorem 
        \ref{lamdarst}, with $f(\cdot)$ replaced by $\lambda I\{ \cdot \le t \}$. The present proof is shorter and given here  for the sake of brevity.
        }
    \end{rem}
    
    \medskip 
    Our final result specializes Theorem \ref{ldp_fixed} to the case $t=1$, where the rate function can be determined explicitly. 
    The     proof    follows by a straightforward calculation  of  $\Lambda_1(\lambda)$ and its convex conjugate.
 
    \begin{cor} The sequence 
        $(Z_n(1))_{n\in \en} $ satisfies an LDP with good rate function
        \begin{align*}
            I(x) = 
            \begin{cases}
                2x - 1 - \log (2x) & x > 0 \\
                \infty &  \text{else}
            \end{cases}.
        \end{align*}
    \end{cor}

\appendix
\section{Auxiliary results} \label{appendix}
 \def\theequation{A.\arabic{equation}}
 \setcounter{equation}{0}

In the proof of the results we make frequent use of the following approximations, which can
   can be derived from the approximations given in \cite{detgam2007}. Throughout this section
  $C$ denotes a positive constant.
    \begin{align}
        \Big |\mathbb{E}\big [\tilde \xi_{n, i}\big ] + 4\log(2) + \frac{1}{2(n - i + 1)}\Big | &\le \frac{C}{(2n - 2i + 1)^2} \label{erwartung} \\
        \Big |\operatorname{Var}\big (\tilde \xi_{n, i}\big ) - \frac{1}{4(n - i + 1)^2}\Big|  &\le \frac{C}{(2n - 2i + 1)^3} \label{var_genau} \\
        |\mathbb{E}[\log(q_{n, i})] + \log(2)| &\le \frac{C}{n - i + 1} \label{erwartung_einfach} \\
        \operatorname{Var}\big (\tilde \xi_{n, i}\big ) &\le \frac{C}{(n - i + 1)^2} \label{var_grob} \\
        \mathbb{E}\big [\big(\tilde \xi_{n, i} - \mathbb{E}[\tilde \xi_{n, i}]\big)^4\big] &\le \frac{C^2}{(n - i + 1)^4} \label{viertes_moment}
    \end{align}
    Also, a direct estimate of the occurring integrals show for $i < 2n$:
    \begin{align}
        \mathbb{E}\left[|\log(q_{2n, i})|^k\right] \le 3\sup \limits_{x \in [0, 1]} |\log(x)|^k x < \infty \label{moment_grob}
    \end{align}
    Lastly, one can prove by differentiation under the integral that for a random variable $X \sim \beta(a, b)$
    \begin{align}
        \operatorname{Var}(\log(X)) = \psi_1(a) - \psi_1(a + b) \label{var_trigamma}
    \end{align}
    where $\psi_1(x) = \frac{d^2}{dx^2} \log(\Gamma(x)) = x^{-1} + O(x^{-2}) \quad (x \to \infty)$ denotes the trigamma function. \\ \\

    We also need to approximate the moments of gamma-distributed random variables. 
    Using the notation $d_i = \frac{\Gamma^{(i)}(k)}{\Gamma(k)}$ we can see that
    \begin{align*}
        \frac{d}{dk}\log(\Gamma(k)) &= d_1 ~=~ \log(k) - \frac{1}{2k} + O(k^{-2}) \\
        \frac{d^2}{dk^2} \log(\Gamma(k)) &= d_2 - d_1^2  ~=~ \frac{1}{k} +  \frac{1}{2k^2} + O(k^{-3}) \\
        \frac{d^4}{dk^4} \log(\Gamma(k)) &= d_4 - 3d_2^2 - 6d_1^4 - 4 d_3 d_1 + 12 d_1^2d_2 ~=~  O(k^{-3})
    \end{align*}
    where the first part of the equations follows from formally differentiating the term, while the second part follows from the approximations 
    of the polygamma functions in \cite{abrsteg1964}. If $X \sim \gamma(k, 1)$, then for $Y = \log(X)$ the following equations hold
    \begin{align}
        \mathbb{E}[Y] &= d_1 = \log(k) - \frac{1}{2k} + O(k^{-2}) \label{mean_gamma}\\
        \operatorname{Var}(Y) &= d_2 - d_1^2  = \frac{1}{k} +  \frac{1}{2k^2} + O(k^{-3}) \label{variance_gamma}\\
        \mathbb{E}[(Y-\mathbb{E}[Y])^4] &= d_4 - 4 d_1d_3 + 6d_1^2d_2 - 3d_1^4 = 3(d_2 - d_1^2)^2 + O(k^{-3})\label{fourth_moment_gamma} \\
            &= \frac{3}{k^2} + O(k^{-3}) \notag
    \end{align}

    \begin{lem} \label{conv_beta}
        Let $X_n \sim \beta(n, n)$. Then 
        \begin{align*}
            \sqrt{n}\left(X_n - \frac{1}{2}\right) \xrightarrow{d} \mathcal{N}(0, \tfrac{1}{8})
        \end{align*}
    \end{lem}
    \begin{proof}
        The density of $\sqrt{n}\left(X_n - \frac{1}{2}\right)$ is given by
        \begin{align}
             &B(n, n)^{-1} \frac{1}{\sqrt{n}}\left(\frac{1}{2} + \frac{x}{\sqrt{n}}\right)^{n - 1} \left(\frac{1}{2} - \frac{x}{\sqrt{n}}\right)^{n - 1} I\Big\{0 < \frac{x}{\sqrt{n}} + \frac{1}{2} < 1\Big\} \nonumber \\
            ={}& B(n, n)^{-1} \frac{1}{\sqrt{n}} \frac{1}{4^{n - 1}} \left(1 - \frac{4x^2}{n}\right)^{n - 1} I\Big\{-\frac{\sqrt{n}}{2} < x < \frac{\sqrt{n}}{2}\Big\}. \label{density_beta}
        \end{align}
        By Stirling's approximation
        \begin{align*}
            B(n, n) = \frac{\Gamma(n)^2}{\Gamma(2n)} \sim \frac{\sqrt{\pi}}{2^{2n - 1} \sqrt{n}}
        \end{align*}
        holds, which implies that the density (\ref{density_beta}) converges pointwise to the density of the $\mathcal{N}(0, \tfrac{1}{8})$-distribution. An application of Scheff\'{e}'s theorem (c.f. \cite{scheffe1947}) yields the desired result.
    \end{proof}
\bigskip

{\bf Acknowledgements.} The authors would like to thank
M. Stein who typed this manuscript with considerable technical
expertise.
The work of H. Dette 
was partially supported by the Deutsche
Forschungsgemeinschaft (DFG Research Unit 1735; DE 502/26-2).
The work of D. Tomecki was supported  by     the Deutsche
Forschungsgemeinschaft (RTG 2131). The authors would like to thank an unknown referee for the careful reading of 
and the constructive comments on an earlier version of this paper.

\bigskip

\bibliography{quellen}

\end{document}